\documentclass[12pt]{article}

\usepackage{amsmath,amsthm,amsfonts,amssymb, color,xcolor,subcaption,graphicx} 

\usepackage[normalem]{ulem}
\usepackage{todonotes, dsfont}
\usepackage{pgfplots}
\usepackage[hidelinks]{hyperref}
\numberwithin{equation}{section}

\newtheorem{theorem}{Theorem}[section]

\newtheorem{proposition}[theorem]{Proposition}
\newtheorem{lemma}[theorem]{Lemma}

\newtheorem{remark}[theorem]{Remark}

\def\cA{\mathcal{A}}
\def\cB{\mathcal{B}}

\def\cF{\mathcal{F}}
\def\cG{\mathcal{G}}
\def\cH{\mathcal{H}}
\def\cI{\mathcal{I}}

\def\cN{\mathcal{N}}

\def\cR{\mathcal{R}}

\def\cS{\mathcal{S}}

\def\bD{\mathbb{D}}
\def\bE{\mathbb{E}}

\def\bP{\mathbb{P}}
\def\bR{\mathbb{R}}

\newcommand{\ud}{\ensuremath{\mathrm{d}}}

\begin{document}

\title{A functional central limit theorem for the KPZ
  equation with narrow wedge initial data}

\author{Raluca M. Balan\footnote{Corresponding author. University of Ottawa, Department of Mathematics
  and Statistics, 150 Louis Pasteur Private, Ottawa, Ontario, K1N 6N5, Canada.
  E-mail address: rbalan@uottawa.ca.} \and Le Chen \footnote{Auburn University,
  Department of Mathematics and Statistics, Office 3147, 305 W. Samford Avenue,
  Auburn, Alabama 36849, USA. E-mail address: le.chen@auburn.edu} \and Jinxin
  Wang\footnote{University of Ottawa, Department of Mathematics and Statistics,
  150 Louis Pasteur Private, Ottawa, Ontario, K1N 6N5, Canada. E-mail address:
  jwang023@uottawa.ca.} }


\maketitle

\begin{abstract}
  \noindent {In this article, we consider the Kardar-Parisi-Zhang (KPZ) equation with narrow wedge initial data. We prove a functional central limit theorem for the spatial average of the solution, which solves an open problem from \cite{gu-pu}. Our main tool is Clark-Ocone formula from Malliavin calculus, combined with an in-depth analysis of the increments of the solution. The main step is a moment bound for time increments of the normalized spatial integral. For this, we introduce a new technique, called a ``common-ray'' change of variables. Using the fact that Malliavin derivative of the solution satisfies a Chapman-Kolmogorov relation,  we reduce the kernel differences at two
  times to spatial increments at one common time.}
\end{abstract}

\noindent {\em MSC 2020:} Primary 60H15; Secondary 60F05, 60H07, 35R60

\vspace{1mm}

\noindent {\em Keywords:} Kardar-Parisi-Zhang equation, narrow wedge
initial data, functional central limit theorem, tightness, stochastic heat
equation, Malliavin calculus



\section{Introduction}

Spatial averages of stochastic partial differential equations (SPDEs)
{are known to exhibit} macroscopic Gaussian fluctuations, {i.e. they satisfy a central limit theorem.} For the
Kardar-Parisi-Zhang (KPZ) equation with narrow wedge initial data,
finite-dimensional convergence is known {from \cite{gu-pu}}, but temporal increments require
separate control to obtain convergence of paths. {The corresponding functional limit theorem has been left open in \cite{gu-pu}. The goal of this article is to} prove the needed tightness
on compact time intervals bounded away from zero.

A starting point for central limit theorems {for spatial averages} is the work
of \cite{HNV20} on the stochastic heat equation (SHE) with flat initial
condition $u(0,x) = 1$ and a Lipschitz noise coefficient $\sigma$. In the linear
case $\sigma(u) = u$, this is the {\em parabolic Anderson model} (PAM):
\begin{equation}\label{PAM}
  \partial_t u(t,x)=\frac{1}{2} \partial_x^2 u(t,x)+u(t,x) \xi(t,x), \quad t>0,x\in \bR, \quad u(0,\cdot)=1.
\end{equation}
The mild solution of \eqref{PAM} is a jointly measurable,
$(\cF_t)$-adapted random field satisfying
\[
  u(t,x)=1+\int_0^t \int_{\bR} p_{t-s}(x-y)u(s,y)\xi(\ud s,\ud y),
\]
where the integral is interpreted in the It\^o sense, $(\cF_t)_{t\ge 0}$ is the
filtration induced by $\xi$ (given by \eqref{def-Ft} below), and $p_t(x)$ is the
heat kernel:
\[
  p_{t}(x)=\frac{1}{\sqrt{2\pi t}} \exp\left(-\frac{|x|^2}{2t}\right).
\]

The authors of \cite{HNV20} studied the asymptotic behaviour as $N \to \infty$
of the {\em spatial average}:
\[
  F_N(t)=\int_0^N \big(u(t,x)-1\big)\ud x.
\]
Their study was motivated by the fact that, after proper normalization, $F_N(t)$
is asymptotically equivalent to $F_N^*(t)$, the spatial average of the solution
$u_N^{*}(t,x)$ of a localized version of the original equation, which has the
property that  $\{u_N^{*}(t,x_i)\}_{i \ge 1}$ are independent if the $x_i$'s
are far from each other.
The precise meaning of this localization was given in \cite{CJK13}. By
writing the interval $[0, N]$ as a union of alternating big and small blocks, we
obtain enough spacing between the big blocks to prove a central limit theorem
(CLT). The small blocks become asymptotically negligible. The normalization is
given by the variance of $F_N(t)$. Since $\{u(t,x)\}_{x \in \bR}$ is strictly
stationary, $\rho_t^{u}(x):={\rm Cov}(u(t,x),u(t,0)) = \int_0^t p_{2(t-s)}(x)
\bE|u(s,0)|^2 \ud s$ (by It\^o's isometry), and
\[
  {\rm Var}(F_N(t))=\int_0^N \int_0^N \rho_t^{u}(x-y)\ud x \ud y=2\int_0^N
    \rho_t^{u}(x)(N-x)\ud x \sim N K_t,
\]
as $N \to \infty$, using the fact that $K_t:=\int_{\bR}\rho_t^u(x)\ud x = \int_0^t \bE|u(s,0)|^2 \ud s < \infty$ since $\bE|u(t,0)|^2
= 2 e^{t/4}\Phi(\sqrt{t/2})$, where $\Phi$ is the standard normal c.d.f. The
novelty is that this CLT is proved in \cite{HNV20} using a combination between
Stein method and Malliavin calculus, which provides a rate for the convergence
in the total variation distance, called a {\em quantitative} CLT. Moreover, the
authors of \cite{HNV20} proved a corresponding functional CLT:
\[
  \frac{1}{\sqrt{N}}F_N(\cdot) \stackrel{d}{\to} \cG(\cdot) \quad \mbox{in} \quad C[0,T],
\]
where $\cG$ is a centered Gaussian process with covariance $\bE[\cG(t)\cG(s)] =
\int_0^{t\wedge s} \bE|u(r,0)|^2 \ud r$. Here $C[a, b]$ is the space
of continuous functions $f:[a, b] \to \bR$ equipped with the topology of uniform
convergence.

\medskip

Article \cite{HNV20} sparked the development of a large area dedicated to
similar QCLTs and FCLTs for a variety of equations and noises. We refer the
reader to \cite{BHWXY, BNQZ, BS26, BY24, BZ24, BNZ, DNZ, ebina23, ebina25,
gu-li20, li-zhang23, NSZ20, NXZ22, NZ20, NZ21} for a sample of references. The
intuition behind these results is the same: the noise may be correlated in
space, but as long as the correlation between the big blocks decays fast enough,
the resulting variables remain weakly dependent. Moreover, the same intuition
applies to a {\em function} $g$ of the solution. In \cite{CKNP22-AIHP}, the
authors proved that
\[
  \frac{1}{N^{d/2}} {\int_{[0,N]^d}} \Big( g(u(t,x)) -\bE[g(u(t,x))]\Big)\ud x
    \stackrel{d}{\longrightarrow} N(0,c_t)
\]
where $u$ is the solution of (SHE) with flat initial condition, driven by a
spatially homogeneous Gaussian noise (with integrable covariance kernel), and
$g$ is a Lipschitz function. This result was extended in \cite{CKNP23-SPDE}
to locally Lipschitz functions $g$, such as $g(u) = \log u$. This
case is particularly important since, by the Cole-Hopf transform,
$h^{\rm flat} = \log u$ is the solution of the {\em Kardar-Parisi-Zhang (KPZ)
equation} introduced in \cite{KPZ}.

The KPZ equation is formally written as:
\begin{equation}\label{KPZ}
  \partial_t h(t,x)=\frac{1}{2} \partial_x^2 h(t,x)+ \frac{1}{2} \big(\partial_x
    h(t,x)\big)^2+\xi(t,x), \quad t>0,x\in \bR,
\end{equation}
and models the growth of a randomly evolving surface, for instance a surface
formed by particles landing and accumulating at random. Its solution $h(t,x)$
represents the height of the surface at time $t$ and position $x$. The diffusion
smooths the surface, the nonlinear term makes the growth depend on the local
slope, and the noise introduces random irregularities. Much of the interest in
the KPZ equation comes from {\em universality}: many systems with very different
microscopic rules are expected to show the same large-scale fluctuation
behaviour. KPZ also reveals connections between random growth, interacting
particle systems, directed polymers, and random matrix theory. But defining the
solution rigorously is challenging: the solution is too rough for
$(\partial_xh)^2$ to make sense, so this term requires re-normalization;
see \cite{hairer13}.

In the present article, we consider the Cole-Hopf transform
\begin{equation}\label{def-h}
  h(t,x)=\log Z(t,x)
\end{equation} of the solution of (PAM) with Dirac delta initial condition:
\begin{equation}\label{SHE}
  \partial_t Z (t,x)=\frac{1}{2}\partial_{x}^2 Z (t,x)+Z(t,x)\xi(t,x), \quad
  t>0,\ x\in\bR \quad \quad \quad Z(0,\cdot)=\delta_0.
\end{equation}

We say that $h = h^{nw}$ is the solution of the KPZ \eqref{KPZ} with {\em narrow
wedge initial data}, interpreted by approximation (since we cannot take the
logarithm of $\delta_0$): if $Z_{\varepsilon}(0,x) = \frac{1}{2 \varepsilon}
\exp(-\frac{|x|}{\varepsilon})$, then $h(0,\cdot) = \lim_{\varepsilon \to 0+}
h_{\varepsilon}(0,\cdot)$, where
\[
  h_{\varepsilon}(0,x)=\log Z_{\varepsilon}(0,x)=-\log(2\varepsilon)-\frac{|x|}{\varepsilon}.
\]
Figure \ref{narrow-wedge-pic} shows some profiles of $h_{\varepsilon}$ which
justify the name: as $\varepsilon$ decreases, the wedge becomes narrowed and its
peak rises.

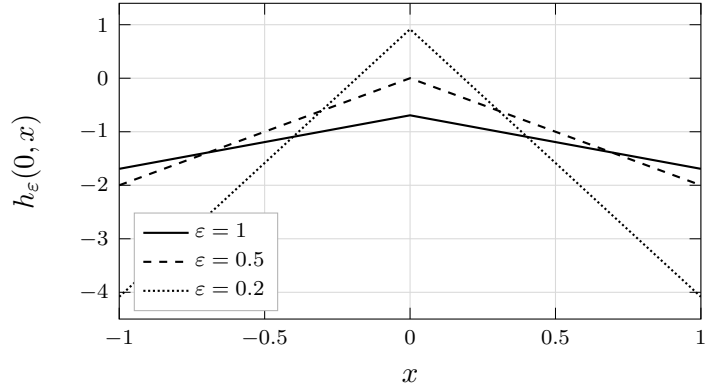
\begin{figure}[htbp]
  \centering
  \begin{tikzpicture}
    \begin{axis}[
      width=0.58\textwidth,
      height=0.36\textwidth,
      xmin=-1, xmax=1,
      ymin=-4.5, ymax=1.4,
      xtick={-1,-0.5,0,0.5,1},
      ytick={-4,-3,-2,-1,0,1},
      xlabel={$x$}, ylabel={$h_\varepsilon(0,x)$},
      label style={font=\small, text=black},
      tick label style={font=\scriptsize, text=black},
      axis line style={black},
      tick style={black},
      grid=major,
      grid style={black!15},
      legend style={
        at={(0.025,0.025)}, anchor=south west,
        font=\scriptsize, text=black,
        fill=white, draw=black!30,
        cells={anchor=west}
      }
    ]
      \addplot[black, thick, domain=-1:1, samples=3]
        {-ln(2)-abs(x)};
      \addlegendentry{$\varepsilon = 1$}
      \addplot[black, thick, dashed, domain=-1:1, samples=3]
        {-2*abs(x)};
      \addlegendentry{$\varepsilon = 0.5$}
      \addplot[black, thick, densely dotted, domain=-1:1, samples=3]
        {-ln(0.4)-5*abs(x)};
      \addlegendentry{$\varepsilon = 0.2$}
    \end{axis}
  \end{tikzpicture}
  \caption{Narrow wedge profiles $h_\varepsilon(0,x) =
    -\log(2\varepsilon)-|x|/\varepsilon$ for $\varepsilon = 1, \,0.5, \,0.2$.}
  \label{narrow-wedge-pic}
\end{figure}

\pagebreak Following Walsh \cite{walsh86}, we interpret the solution
of \eqref{SHE} in the mild sense, i.e. we assume that $\{Z(t,x);t \ge 0, x \in
\bR\}$ is a measurable and $(\cF_t)$-adapted random field which satisfies:
\begin{equation}\label{eq-Z}
  Z(t,x)=p_t(x)+\int_0^t \int_{\bR} p_{t-s}(x-y) Z(s,y) \xi(\ud s,\ud y),
\end{equation}
where the stochastic integral is interpreted in the It\^o sense. The
normalization below removes the deterministic spatial profile and allows us to
use stationarity. Indeed, the process
\[
  \overline{Z}(t,x)=\frac{Z(t,x)}{p_t(x)}, \quad x \in \bR
\]
is strictly stationary (see Proposition 1.4 of \cite{ACQ11}). Moreover,
$\overline{Z}$ satisfies:
\[
  \overline{Z}(t,x)=1+\int_0^t \int_{\bR} p_{\frac{s(t-s)}{t}}\left( y-\frac{s}{t}x\right)
    \overline{Z}(s,y)\xi(\ud s,\ud y),
\]
using property \eqref{def-q} below of the heat kernel. In \cite{CKNP22-JFA}, the
authors proved that its spatial average $\overline{S}_N(t) = \int_0^N
\big(\overline{Z}(t,x)-1\big)\ud x$ satisfies the following FCLT:
\[
  \frac{1}{\sqrt{N \log N}} \overline{S}_N(\cdot) \stackrel{d}{\longrightarrow} \sqrt{2}B(\cdot)
    \quad \mbox{in} \quad C[0,T],
\]
where $B$ is a standard Brownian motion. The normalization comes from the fact
that ${\rm Var}\big(\overline{S}_N(t) \big) \sim 2 t N \log N$, according to
Proposition 4.1 of \cite{CKNP22-JFA}. Alternatively, this follows from
stationarity and the fact that
\begin{equation}\label{E:rho}
  \rho_t^{\overline{Z}}(x):={\rm Cov}(\overline{Z}(t,x),\overline{Z}(t,0)) =\int_0^t
    \bE|\overline{Z}(s,0)|^2 \, p_{\frac{2s(t-s)}{t}}\Big(\frac{s}{t} x\Big) \ud s \sim
    \frac{t}{x} \quad \mbox{as} \quad x\to \infty,
\end{equation}
where the {second-moment formula (2.31) of \cite{chen-dalang15}
  (with $\nu = \lambda = 1$ in their notation)} gives
\[
  m(s):=\bE|\overline{Z}(s,0)|^2
  =1+\sqrt{\pi s}\,e^{s/4}\Phi(\sqrt{s/2}).
\]
In particular, $m(s) \ge 1$ and $m(s) \to 1$ as $s \downarrow 0$. The factor
$s/t$ in the kernel argument is essential: by Tonelli's theorem and the change
of variables $z = sx/t$,
\[
  \int_0^\infty\rho_t^{\overline{Z}}(x)\,\ud x
  =\frac{t}{2}\int_0^t\frac{m(s)}s\,\ud s=\infty.
\]
The covariance tail and variance asymptotic are {given by
Proposition \ref{P:normalized-SHE-covariance} below.}

\medskip

There is a natural connection between $\overline{Z}$ and $h$: since $\log
\overline{Z}(t,x) = h(t,x)-\log p_t(x)$,
\[
  h(t,x)-\bE[h(t,x)]=\log \overline{Z}(t,x)-\bE[\log \overline{Z}(t,x)], \quad x\in \bR
\]
is strictly stationary. It turns out that the spatial average of $h$:
\begin{equation}\label{def-SN}
  S_N(t)=\int_0^N \big(h(t,x)-\bE[h(t,x)]\big)\ud x,
\end{equation}
has a similar asymptotic behaviour to $\overline{S}_{N}(t)$. The recent preprint
\cite{gu-pu} proved that
\begin{equation}\label{var-asy}
  {\rm Cov}(h(t,x),h(t,0))\sim \frac{t}{x} \quad \mbox{as} \quad x \to \infty,
\end{equation}
and consequently, ${\rm Var}\big(S_N(t) \big) \sim 2t N \log N$ as $N \to
\infty$. This is very different from the case of the flat initial data: Theorem
1.1 of the recent article \cite{chen-jimenez26} showed that
\[
  \mathrm{Cov}\big(h^{\rm flat}(t,x),h^{\rm flat }(t,0)\big) \sim
  \frac{2\kappa(t)}{\sqrt{\pi}}\,t^{3/2}\,|x|^{-2}
  \exp\Big(-\frac{x^{2}}{4t}\Big),
  \quad \quad \mbox{as} \ x\to\infty.
\]
for a certain constant $\kappa(t) > 0$. Since this covariance is integrable,
${\rm Var}\big(S_N^{\rm flat}(t) \big) \sim C_t N$, where $S_N^{\rm flat}(t)$ is
defined as \eqref{def-SN} with $h$ replaced by $h^{\rm flat}$. In fact, due to
the earlier work of \cite{CKNP23-SPDE}, it was already known that $S_N^{\rm
flat}$ satisfies the functional CLT: (see Theorem 2.3(ii) {\em
ibid.}),
\[
  \frac{1}{\sqrt{N}}S_{N}^{\rm flat}(\cdot) \stackrel{d}{\to} \cG(\cdot) \quad \mbox{in} \quad C[0,T],
\]
where $\cG$ is a centered Gaussian process with 
\[
  \bE[\cG(t)\cG(s)]
  =\int_{\bR}\mathrm{Cov}\big(h^{\rm flat}(t,x),h^{\rm flat}(s,0)\big)\,\ud x.
\]

\medskip

In addition to \eqref{var-asy}, the authors of \cite{gu-pu} have also proved
that as $N \to \infty$,
\[
  \frac{1}{\sqrt{N \log N}} S_N(\cdot) \stackrel{f.d.d.}{\longrightarrow} \sqrt{2}B(\cdot),
\]
where $\stackrel{f.d.d.}{\longrightarrow}$ denotes convergence in the sense of
finite-dimensional distributions. They have left as an open problem the
extension of this result to a functional CLT; see Remark 4.3. {\em ibid}.
{The goal of the present article is to} establish this process-level convergence on every compact time
interval bounded away from zero.

\medskip

The following theorem is the main result of the present article.

\begin{theorem}\label{T:FCLT}
  For any $N \ge 2$, the process $S_N(\cdot) = \{S_N(t)\}_{t>0}$ has a
  locally $\gamma$-H\"older continuous modification of order $\gamma
  \in (0,\frac{1}{4})$, which we denote also by $S_N(\cdot)$. Moreover, as $N
  \to \infty$,
  \[
    \frac{1}{\sqrt{N \log N}} S_N(\cdot) \stackrel{d}{\longrightarrow} \sqrt{2}B(\cdot), \quad
      \mbox{in $C[\varepsilon,T]$},
  \]
  for any $0 < \varepsilon < T$, where $(B_t)_{t\ge 0}$ is a standard Brownian
  motion, and $\stackrel{d}{\longrightarrow}$ denotes convergence in
  distribution.
\end{theorem}

We state Theorem \ref{T:FCLT} on $[\varepsilon, T]$ because our
increment estimates depend on $\varepsilon$ and are not uniform as $\varepsilon
\downarrow 0$. For the linear observable $\overline{Z}-1$, {Lemma 2.4 of \cite{CKNP22-JFA} and Proposition 6.1 {\em ibid.} combine} a quantitative small-time bound on its
normalized spatial average with increment estimates whose deterioration near
zero is explicitly controlled. Our observable is
$\log\overline{Z}-\bE[\log\overline{Z}]$. Pointwise moment control, including
negative moments of $\overline{Z}$, does not by itself provide the required
spatial-averaging gain. We do not establish the corresponding uniform small-time
spatial-average estimate and interpolation for this observable; our result
therefore leaves convergence on $[0, T]$ open. 

The limiting finite-dimensional distributions are already supplied
by \cite{gu-pu}. To prove Theorem \ref{T:FCLT}, we establish a moment bound for
time increments of the normalized spatial integral, uniformly in the averaging
length. The obstruction is that temporal increments of the normalized
narrow-wedge SHE are not stationary in space, so the estimates used for flat
initial data do not transfer directly; see {Remark 4.3 of \cite{gu-pu}}.

{
As in  \cite{gu-pu}, we will use  the Clark-Ocone formula to express $X_N(t)$ as a stochastic integral with respect to $\xi(\ud s,\ud y)$ of spatial $\ud x$ integral on $[0,N]$ of a product of the from $\cA(t,x;s,y)q_{s}^{t,x}(y)$; see relation \eqref{S-Aq} below. But unlike \cite{gu-pu}, we will {\em not} consider the projection of $X_N(t)$ on the first Winer chaos space (which is the same stochastic integral, but with the integrand replaced by its expectation). 

The main difficulty is to reduce the increment $\cA(t_2,x;s,y)q_{s}^{t_2,x}(y)-\cA(t_1,x;s,y)q_{s}^{t_1,x}(y)$ to an increment which contains the {\em same} variable $t_1$. For this, we will use a new idea, based on rewriting the two $dx$ integrals using the changes of variables $x=t_1 u$ and $x=t_2u$. We call this the ``common-ray procedure'', where $u$ is the ``ray variable''. See relation \eqref{com-ray}. }

The key ingredient is the fact that the Malliavin derivative $\cR(t,x;s,y)=D_{s,y}h(t,x)$ satisfies a Chapmam-Kolmogorov relation; see relation \eqref{CK-R}.  This
reduces its difference at two times to a spatial increment at one common time; see relation \eqref{diff-R}.
To obtain tightness, we estimate the resulting bulk and boundary terms separately.


\medskip

The remainder of this article is organized in two sections: in
Section \ref{section-prelim}, we introduce some preliminary material, while in
Section \ref{section-FCLT} we present the proof of Theorem \ref{T:FCLT}.
Appendix \ref{S:appendix} contains some auxiliary material which is
used in the paper.

\medskip

We end the introduction with a few words about the notation. For any $p \ge 1$,
we denote by $\|\cdot\|_p$ the norm in $L^p(\Omega)$. We denote by
$C_{a,b,\ldots}$ a generic positive constant which depends on parameters $a, b,
\ldots$ (listed as lower indices), which may be different in each of its
appearances.

\bigskip

\section{Preliminaries}
\label{section-prelim}

In this section, we include some preliminary material which will be used below.

We assume that $\xi$ is defined on a complete probability space
$(\Omega,\cF,\bP)$. We denote by $(\cF_t)_{t\ge 0}$ the (augmented) filtration
induced by $\xi$:
\begin{equation}\label{def-Ft}
  \cF_t=\sigma\{ \xi([0,s] \times A); s \in [0,t], A \in \cB_b(\bR) \} \vee \cN,
\end{equation}
where $\cB_b(\bR)$ is the class of bounded Borel sets in $\bR$, and $\cN$ are
$\bP$-negligible events.

\medskip

A random process $u = \{u(t,x);t \ge 0, x \in \bR\}$ is {\em $(\cF_t)$-adapted}
if $u(t,x)$ is $\cF_t$-measurable for any $t > 0$ and $x \in \bR$. The process
$u$ is {\em measurable} if the map $(\omega,t,x) \mapsto u(\omega,t,x)$ is $\cF
\otimes \cB(\bR_{+}) \otimes \cB(\bR)$-measurable.

\subsection{Malliavin calculus}

In this section, we include some basic facts on Malliavin calculus with respect
to the isonormal Gaussian process $\{\xi(\varphi)\}_{\varphi \in \cH}$ induced
by the space-time Gaussian white noise:
\[
  \bE[\xi(g)\xi(h)]=\langle g,h\rangle_{\cH} \quad \mbox{with} \quad \cH=L^2(\bR_{+} \times \bR).
\]

\medskip

$\bullet$ {\bf Malliavin derivative.} Let $\cS$ be the class of ``smooth''
random variables:
\begin{equation}\label{form-F}
  F=f(\xi(\varphi_1),\ldots, \xi(\varphi_n)),
\end{equation} where $f \in C_{b}^{\infty}(\bR^n)$, $\varphi_i \in \cH$,
$n \ge 1$, and $C_b^{\infty}(\bR^n)$ is the class of bounded
$C^{\infty}$-functions on $\bR^n$, whose partial derivatives of all orders are
bounded. The {\em Malliavin derivative} of $F$ of the form (\ref{form-F}) is the
$\cH$-valued random variable given by:
$$
  D_{t,x}F:=\sum_{i=1}^{n}\frac{\partial f}{\partial x_i}(\xi(\varphi_1),\ldots,
  \xi(\varphi_n))\varphi_i(t,x), \quad \mbox{for $t>0$ and $x\in \bR$}.
$$
We endow $\cS$ with the norm $\|F\|_{\bD^{1,2}} =
\big(\bE|F|^2+\bE\|DF\|_{\cH}^2\big)^{1/2}$. The operator $D$ can be extended
to the space $\bD^{1,2}$, the completion of $\cS$ with respect to $\|\cdot
\|_{\bD^{1,2}}$.

If $F \in \bD^{1,2}$ is $\cF_t$-measurable, then
\begin{equation}\label{Mal-zero}
  D_{r,z}F=0 \quad \mbox{for almost every $(r,z)\in(t,\infty)\times\bR$, a.s.}.
\end{equation}

\medskip

$\bullet$ {\bf It\^o integral.} The It\^o integral $\int_0^t \int_{\bR}
g(s,x)\xi(\ud s,\ud x), t \ge 0$ is a martingale defined
for any {predictable} process $g$ such that $\bE\int_0^t \int_{\bR}
g^2(s,x)\ud x\ud s < \infty$ for any $t > 0$. By
Burkholder-Davis-Gundy and Minkowski's inequalities, for any $p \ge 2$,
\begin{equation}\label{BDG}
  \left\| \int_0^T \int_{\bR} g(t,x)\xi(\ud t,\ud x) \right\|_p^2 \le z_p^2
    \int_0^T \int_{\bR}\|g(t,x)\|_p^2 \ud t\ud x,
\end{equation}
where $z_p > 0$ is a constant depending on $p$.

\medskip

$\bullet$ {\bf Skorohod integral.} The {\em divergence operator} $\delta$ is the
adjoint of the operator $D$. If $u \in {\rm Dom} (\delta)$, then $\delta(u) \in
L^2(\Omega)$ is characterized by the duality relation:
\begin{equation}\label{duality}
  \bE(F \delta(u))=\bE\left(\langle DF,u \rangle_{\cH}\right), \quad
  \forall F \in \bD^{1,2}.
\end{equation}
In particular, $\bE[\delta(u)] = 0$. For any $u \in \mbox{Dom} (\delta)$, we say
that $\delta(u)$ is the {\em Skorohod integral} of $u$. If $g$ is
predictable and $\bE\int_0^\infty\int_{\bR}|g(s,x)|^2\,\ud x\,\ud s < \infty$, then $g \in \mathrm{Dom}(\delta)$ and
$\delta(g) = \int_0^\infty\int_{\bR}g(s,x)\xi(\ud s,\ud x)$.

\medskip

We will use the following results. We refer the reader to \cite{nualart06} for
more details.

\medskip

(i) {\em Heisenberg commutation principle:} we use the {following
result (Proposition 1.3.8 of \cite{nualart06})}. Suppose that $u \in \bD^{1,2}(\cH)$
is supported in $[0, T]\times\bR$. Assume that $D_{t,x}u \in
\mathrm{Dom}(\delta)$ for almost every $(t,x)$, and that
$\{\delta(D_{t,x}u)\}_{t,x}$ has a jointly measurable version in
$L^2(\Omega;\cH)$. Then $\delta(u) \in \bD^{1,2}$ and, almost everywhere in
$(t,x)$,
\begin{equation}\label{heisen}
  D_{t,x} (\delta(u))=u(t,x)+\int_{0}^{T} \int_{\bR}D_{t,x}u(s,y) \xi(\delta s,\delta y).
\end{equation}

(ii) {\em Chain rule:} if $F \in \bD^{1,2}$ and $\varphi \in C^1(\bR)$ 
has bounded derivative, then $\varphi(F) \in \bD^{1,2}$ and
\begin{equation}\label{chain}
  D \varphi(F)=\varphi'(F) DF.
\end{equation}

(iii) {\em Clark-Ocone formula:} if $F \in \bD^{1,2}$ and $F$ is
$\cF_T$-measurable, then
\begin{equation}\label{clark-ocone}
  F=\bE(F)+\int_0^T \int_{\bR} \bE[D_{t,x}F|\cF_t] \xi(\ud t,\ud x).
\end{equation}
 Here and below, conditional expectations used as stochastic integrands
denote their predictable versions, defined up to equality in
$L^2(\Omega\times[0,T]\times\bR)$.

\subsection{Green function of SHE}

In this section, we include some background material about a random process
$\cG$ which is very helpful for our calculations. This process appeared also
in \cite{alberts} and \cite{gu-pu}.

\medskip

Let $s \ge 0$ and $y \in \bR$ be fixed. We denote by $\{\cG(t,x;s,y);t > s,
{x} \in \bR\}$ the solution of the following (time-delayed) stochastic
heat equation:
\[
  \begin{cases}
    \displaystyle
    \partial_t \cG (t,x;s,y)=\frac{1}{2}\partial_x^2 \cG (t,x;s,y)+\cG(t,x;s,y)\xi(t,x), \
    t>s,\ x\in\bR, \\
    \cG(s,\cdot;s,y)=\delta_y.
  \end{cases}
\]
This means that $\{\cG(t,x;s,y);t > s, x \in \bR\}$ is a measurable
and $(\cF_t)_{t\ge s}$-adapted random field such that for any $t >
s$ and $x \in \bR$, with probability 1,
\begin{equation}\label{G-eq}
  \cG(t,x;s,y)=p_{t-s}(x-y)+\int_s^t \int_{\bR} p_{t-r}(x-z) \cG(r,z;s,y) \xi(\ud r,\ud z).
\end{equation}
Taking $s = 0$ and $y = 0$ in \eqref{G-eq} gives exactly \eqref{eq-Z}, so
$\cG(\cdot,\cdot;0,0)$ and $Z$ are both random field solutions of \eqref{SHE}.
Such a solution is unique, by Theorem 2.4 of \cite{chen-dalang15}. Therefore,
for any $t > 0$ and $x \in \bR$, with probability 1,
\begin{equation}\label{GZ}
  \cG(t,x;0,0)=Z(t,x).
\end{equation}

Multiplying \eqref{G-eq} by $Z(s,y)$ and using a property of the It\^o integral
which states that an $\cF_s$-measurable factor can be moved inside an It\^o
integral starting at $s$, we obtain:
\[
  \cG(t,x;s,y)Z(s,y)=p_{t-s}(x-y)Z(s,y)+\int_s^t \int_{\bR} p_{t-r}(x-z) \cG(r,z;s,y)
    Z(s,y)\xi(\ud r,\ud z).
\]

On the other hand, we import Malliavin differentiability of $Z(t,x)$
and the derivative equation below from Section 2 of \cite{gu-pu}. The derivative
vanishes for $s > t$ by \eqref{Mal-zero}, and for $0 < s < t$, with the identity
understood almost everywhere in $(s,y)$, we have:
\[
  D_{s,y}Z(t,x)=p_{t-s}(x-y)Z(s,y)+\int_s^t \int_{\bR} p_{t-r}(x-z) D_{s,y}Z(r,z) \xi(\ud r,\ud z).
\]
The two right-hand sides above are the mild formulation of the same equation,
namely
\[
  \begin{cases}
    \displaystyle \partial_t V(t,x)=\frac12\partial_x^2 V(t,x)+V(t,x)\xi(t,x),
    \quad t>s,\ x\in\bR,\\
    V(s,\cdot)=Z(s,y)\delta_y,
  \end{cases}
\]
whose random field solution is unique by Theorem 2.4 of \cite{chen-dalang15}.
The initial datum is $\cF_s$-measurable and the noise on $(s,\infty)$ is
independent of $\cF_s$, so that theorem applies conditionally on $\cF_s$.
Therefore, for ${0 < s < t}$ and $x \in \bR$, with the Malliavin
identity understood almost everywhere in $(s,y)$,
\begin{equation}\label{DZ-prod}
  D_{s,y}Z(t,x)=\cG(t,x;s,y)Z(s,y).
\end{equation}

\begin{remark}
  {\rm  The same factorization argument applies to other initial data
  and semigroups whenever the corresponding mild Green kernel, Malliavin
  differentiability and uniqueness for the random initial-value equation above
  are available. These conditions must be checked for each such extension. }
\end{remark}

The family $\{\cG(t,x;s,y);0 \le s < t,\ x,y \in \bR\}$ is called the
{\em Green function of SHE}. We define the {\em scaled} Green function of SHE
by:
\[
  \overline{\cG}(t,x;s,y)=\frac{\cG(t,x;s,y)}{p_{t-s}(x-y)} \quad \quad \quad  0\le s<t, \ x,y\in \bR.
\]
At the diagonal we set $\overline{\cG}(s,x;s,y) = 1$ by convention;
the quotient above is used only when $s < t$. {We do not assign} any pointwise value to
the unnormalized kernel at the diagonal.

 We import the following uniform moment bound from Lemma 3.3 of \cite{chen-jimenez26}, which follows from Lemma 3.2 of \cite{alberts} together with
time-translation and shear invariance. For any $k \in \bR$ and $T > 0$,
\begin{equation}\label{barG}
  \sup_{\substack{0\le s \le t \le T \\ x,y \in \bR}} \bE[\overline{\cG}(t,x;s,y)^k]<\infty.
\end{equation}

By Theorem 2.19.(viii) of \cite{COTX}, $\cG$ satisfies the following
Chapman-Kolmogorov relation: with probability 1, for any $0 \le s < r < t$ and
$x, y \in \bR$,
\begin{equation}\label{CK-G}
  \cG (t,x;s,y)=\int_{\bR} \cG (t,x;r,z) \cG (r,z;s,y)\ud z.
\end{equation}

\section{Proof of Theorem \ref{T:FCLT}}
\label{section-FCLT}

In this section, we present the proof of Theorem \ref{T:FCLT}. We define
\[
  X_N(t)=\frac{1}{\sqrt{N \log N}} S_N(t).
\]

We fix $0 < \varepsilon < T$. We will prove that for any $N \ge 2$, $p \ge 2$,
$\alpha \in (0,\frac{1}{2})$ and $t_1, t_2 \in [\varepsilon, T]$,
\begin{equation}\label{main-est}
  \bE|X_{N}(t_2)-X_N(t_1)|^p \le C_{p,\alpha,\varepsilon, T} |t_2-t_1|^{p\alpha/2}.
\end{equation}

Choose $p > 4$ and $2/p < \alpha < 1/2$. By the Kolmogorov--Chentsov
criterion (Theorem 3.23 of \cite{kallenberg02}), \eqref{main-est} gives a continuous
modification with every H\"older exponent $\gamma < \alpha/2-1/p$ on
$[\varepsilon, T]$. On the intervals $[1/m, m]$, $m \ge 2$, the
continuous versions agree almost surely on overlaps: they agree at all rational
times there, and continuity extends the equality. They therefore define one
locally H\"older continuous modification on $(0,\infty)$. Choosing $p$
arbitrarily large and $\alpha$ arbitrarily close to $1/2$ proves the asserted
exponents.

The laws of $X_N(\varepsilon)$ are tight by the finite-dimensional convergence
of Theorem 1.2 of \cite{gu-pu}. Together with the uniform increment
estimate \eqref{main-est}, whose exponent $p\alpha/2$ exceeds one, the tightness
criterion (Corollary 16.9 of \cite{kallenberg02}) gives tightness in $C[\varepsilon,
T]$. Every subsequential limit has the finite-dimensional laws of $\sqrt{2} B$
by the same imported theorem. Since evaluations at a countable dense set
determine the law on this continuous-path space, the limit is unique. This
proves convergence of the full family.


The remaining part of the article is dedicated to the proof of \eqref{main-est}.

\subsection{Clark-Ocone representation}

In this section, we give a representation of $X_N(t)$, derived from {
the} Clark-Ocone formula \eqref{clark-ocone}.

\medskip

Recall definition \eqref{def-h} of $h$.  We first justify applying the
Malliavin chain rule to the logarithm. For fixed $t > 0$ and $x \in \bR$, strict
positivity \cite{alberts} and the imported positive and negative moments
in \eqref{barG} imply $\log Z(t,x) \in L^2(\Omega)$. By \eqref{DZ-prod},
Gaussian factorization and generalized H\"older,
\[
  \left\|\frac{D_{s,y}Z(t,x)}{Z(t,x)}\right\|_2
  \le C_T p_{s(t-s)/t}(y-sx/t),\qquad 0<s<t.
\]
The square of this bound is integrable in $(s,y)$, since its spatial integral is
a constant times $[s(t-s)/t]^{-1/2}$. Choose smooth truncations $\phi_n$ of
$\log$ that agree with it on $[1/n, n]$, satisfy $|\phi_n'(z)| \le C/z$ for $z
> 0$, and $|\phi_n(z)| \le C+|\log z|$ there. For $n \ge 2$, choose $\chi_n
\in C_c^\infty((0,\infty))$ with $0 \le \chi_n \le 1$ and $\chi_n = 1$ on
$[1/n, n]$, and set $\phi_n(z) = \int_1^z\chi_n(r)r^{-1}\,\ud r$ for
$z > 0$. Extend $\phi_n$ constantly to $z \le 0$; this is smooth because
$\chi_n$ vanishes near zero. Then $|\phi_n(z)| \le |\log z|$ and $|\phi_n'(z)|
\le 1/z$ for $z > 0$, as required. The bounded-derivative chain rule applies to
$\phi_n(Z(t,x))$. Dominated convergence gives convergence of the functions in
$L^2(\Omega)$ and of their derivatives in $L^2(\Omega;\cH)$ to $\log Z(t,x)$ and
$DZ(t,x)/Z(t,x)$, respectively. Closedness of $D$ proves $h(t,x) \in \bD^{1,2}$
and the required chain rule.

Applying the chain rule \eqref{chain}, followed by the product
decomposition \eqref{DZ-prod} of $DZ$, the initial-value relation \eqref{GZ},
and the definition of $\overline{\cG}$, we get:
\begin{align}
  \nonumber
  D_{s,y} h(t,x)&=\frac{D_{s,y} Z(t,x)}{Z(t,x)}=\frac{\cG(t,x;s,y)Z(s,y)}{Z(t,x)}
  =\frac{\cG(t,x;s,y)\cG(s,y;0,0)}{\cG(t,x;0,0)} \\
  \label{Dh}
  &=\frac{p_{t-s}(x-y)p_s(y)}{p_t(x)} \cdot
    \frac{\overline{\cG}(t,x;s,y)\overline{\cG}(s,y;0,0)}{\overline{\cG}(t,x;0,0)}.
\end{align}

We use the following property of the heat kernel: for any $t, s > 0$ and $x, y
\in \bR$,
\[
  p_t(x) p_s(y)=p_{t+s}(x-y) p_{\frac{ts}{t+s}} \left( \frac{sx+ty}{t+s}\right),
\]
(see e.g. Lemma A.4 of \cite{chen-dalang15}). It follows that
\begin{equation}\label{def-q}
  \frac{p_{t-s}(x-y)p_s(y)}{p_t(x)}=p_{\frac{s(t-s)}{t}}\Big(y-\frac{s}{t}x\Big)=: q_{s}^{t,x}(y).
\end{equation}

With the Clark-Ocone formula \eqref{clark-ocone} in mind, we write
\[
  \bE[D_{s,y} h(t,x)|\cF_s]=\cA(t,x;s,y) q_{s}^{t,x}(y),
\]
where
\[
  \cA(t,x;s,y)=\bE\Big[\frac{\overline{\cG}(t,x;s,y) \overline{\cG}(s,y;0,0) }{\overline{\cG}(t,x;0,0)}|\cF_s\Big].
\]
We say that $\cA$ is the {\em Clark--Ocone kernel}; see also (3.1)
of \cite{gu-pu}. Therefore,
\begin{align*}
  h(t,x)-\bE[h(t,x)] 
  &= \int_0^t \int_{\bR} \cA(t,x;s,y) q_{s}^{t,x}(y)\xi(\ud s,\ud y).
\end{align*}

The moment bound of Lemma \ref{A-bound}, proved below using only the
imported Green-kernel moments, gives the stochastic Fubini integrability
condition. Indeed, for fixed $t > 0$,
\[
  \int_0^N\left(\int_0^t\int_{\bR}
    \|\cA(t,x;s,y)q_s^{t,x}(y)\|_2^2\,\ud y\,\ud s\right)^{1/2}\ud x
  \le C_T N\left(\int_0^t
    \left(\frac{s(t-s)}t\right)^{-1/2}\ud s\right)^{1/2}<\infty.
\]
Here $\int_{\bR}p_a(y)^2\ud y = (4\pi a)^{-1/2}$, and both endpoint
singularities in the time integral are integrable. To choose the
conditional integrands jointly in $x$, we take the predictable projection in the
product space $L^2([0,N]\times[0,t]\times\bR\times\Omega)$, with Lebesgue
measure in the first three variables. The spatial parameter is retained in this
projection. The displayed bound, uniformly for $x \in [0, N]$, also gives
finiteness of this product-space norm. For almost every $x$, the projected
integrand is the predictable conditional expectation in the Clark--Ocone
formula. Thus, the jointly measurable predictable integrand satisfies the
stochastic Fubini condition displayed above. Therefore, by the stochastic
Fubini theorem,
\begin{equation}\label{S-Aq}
  X_N(t)=\frac{1}{\sqrt{N \log N}}\int_0^t \int_{\bR} \left(\int_0^N  \cA(t,x;s,y)
    q_{s}^{t,x}(y)\ud x \right) \xi(\ud s,\ud y).
\end{equation}

\begin{lemma}\label{A-bound}
  For any $p \ge 1$ and $T > 0$,
  \[
    \sup_{\substack{ 0<s<t\le T \\ x,y \in \bR}} \|\cA(t,x;s,y)\|_p <\infty.
  \]
\end{lemma}

\begin{proof}
  By Jensen's inequality for conditional expectation and generalized H\"older's
  inequality,
  \begin{align*}
    \|\cA(t,x;s,y)\|_p & \le \left\| \frac{\overline{\cG}(t,x;s,y) \overline{\cG}(s,y;0,0)
      }{\overline{\cG}(t,x;0,0)} \right\|_p \\
    & \le \|\overline{\cG}(t,x;s,y)\|_{3p} \|\overline{\cG}(s,y;0,0) \|_{3p}
    \|\overline{\cG}(t,x;0,0)^{-1}\|_{3p}.
  \end{align*}
  The conclusion follows by relation \eqref{barG}. In particular, the
  inverse factor uses the moment estimate at order $-3p$; the first two
  factors use order $3p$.
\end{proof}

\subsection{Increment analysis}

This section is the core of the proof of Theorem \ref{T:FCLT}, and consists of a
detailed analysis of the increments of $X_N$.

\medskip

Let $0 < \varepsilon \le t_1 < t_2 \le T$ be arbitrary. Using \eqref{S-Aq}, we
write
\begin{align}
  \nonumber
  & X_N(t_2)-X_N(t_1) =\frac{1}{\sqrt{N \log N}}\int_{t_1}^{t_2}  \int_{\bR} \left(\int_0^N
    \cA(t_2,x;s,y) q_{s}^{t_2,x}(y) \ud x \right)
  \xi(\ud s,\ud y)+\\
  \nonumber
  & \quad \quad \frac{1}{\sqrt{N \log N}} \int_0^{t_1} \int_{\bR} \left(\int_0^N
    \Big(\cA(t_2,x;s,y) q_{s}^{t_2,x}(y)-\cA(t_1,x;s,y) q_{s}^{t_1,x}(y) \Big) \ud x
    \right) \xi(\ud s,\ud y)\\
  \label{XN-dec}
  & \quad \quad =:I_N(t_1,t_2)+J_N(t_1,t_2).
\end{align}

We study separately the moments of $I_N$ and $J_N$. We begin with $I_N$.

\begin{proposition}\label{IN}
  For any $p \ge 2$, $N \ge 2$, and $0 < \varepsilon \le t_1 < t_2 \le T$,
  \[
    \bE|I_N(t_1,t_2)|^p \le C_{p,\varepsilon,T} (t_2-t_1)^{p/2}.
  \]
\end{proposition}

\begin{proof}
  By inequality \eqref{BDG}, Minkowski's inequality, and Lemma \ref{A-bound},
  \begin{align}
    \nonumber
    \|I_N(t_1,t_2)\|_p^2 & \le \frac{z_p^2}{N \log N} \int_{t_1}^{t_2} \int_{\bR} \left\|\int_0^N
      \cA(t_2,x;s,y) q_{s}^{t_2,x}(y) \ud x \right\|_p^2 \ud y\ud s
      \\
    \nonumber
    & \le \frac{z_p^2}{N \log N} \int_{t_1}^{t_2}  \int_{\bR} \left( \int_0^N \|\cA(t_2,x;s,y)\|_p
      q_{s}^{t_2,x}(y) \ud x \right)^2 \ud y\ud s \\
    \label{IN-bound}
    & \le   \frac{C_{p,T}}{N \log N} \int_{t_1}^{t_2} \int_{\bR} \left(\int_0^N q_{s}^{t_2,x}(y)
      \ud x \right)^2 \ud y\ud s.
  \end{align}
  Using definition \eqref{def-q} of $q_{s}^{t,x}$ and the semigroup property
  $p_t*p_s = p_{t+s}$, we see that
  \begin{align}
    \nonumber
    \int_{\bR}\left(\int_0^N q_{s}^{t,x}(y) \ud x \right)^2 \ud y &
      =\int_0^N \int_0^N p_{\frac{2s(t-s)}{t}} \left(\frac{s}{t}(u-v)\right) \ud u
      \ud v \\
    \label{q-int}
    & \le \int_0^N \int_{\bR} p_{\frac{2s(t-s)}{t}} \left(\frac{s}{t}(u-v)\right) \ud u
      \ud v =\frac{Nt}{s} \le \frac{NT}{\varepsilon}.
  \end{align}
  Therefore, 
  \[
    \|I_N(t_1,t_2)\|_p^2
    \le\frac{C_{p,T}T}{\varepsilon\log N}(t_2-t_1)
    \le C_{p,\varepsilon,T}(t_2-t_1),
  \]
  since $N \ge 2$. Raising this bound to the power $p/2$ proves the claim.
\end{proof}

It remains to estimate $\bE|J_N(t_1,t_2)|^p$. For this, we may try to add and
subtract $\cA(t_1,x;s,y)q_s^{t_2,x}(y)$. As in the proof of
Proposition \ref{IN}, this would lead to $\|\cA(t_2,x;s,y)-\cA(t_1,x;s,y)\|_p$,
and subsequently to
$\|\overline{\cG}(t_2,x;s,y)-\overline{\cG}(t_1,x;s,y)\|_{4p}$. But this
approach is not feasible. See also Remark 4.3 of \cite{gu-pu}. Therefore, we
pursue a different method here.

\medskip

Using the change of variable $x = t u$, we write:
\begin{equation}\label{com-ray}
  \int_0^N \cA(t,x;s,y)q_s^{t,x}(y)\ud x=\int_0^{N/t}t \cA(t,tu;s,y) q_s^{t,tu}(y)\ud u.
\end{equation}
We call $u$ the {\em ray variable}. We apply this for $t = t_1$ and $t = t_2$.
This leads to the following ``common-ray expression'' for $J_N$:
\begin{align}
  \nonumber
  J_N(t_1,t_2)&= \frac{1}{\sqrt{N \log N}}\int_0^{t_1} \int_{\bR}  \int_0^{N/t_2}t_2
    \cA(t_2,t_2u;s,y) q_s^{t_2,t_2u}(y)\ud u \xi(\ud s,\ud y)\\
  \label{dec-JN}
  &-\frac{1}{\sqrt{N \log N}}\int_0^{t_1} \int_{\bR} \int_0^{N/t_1}t_1 \cA(t_1,t_1u;s,y)
    q_s^{t_1,t_1u}(y)\ud u \xi(\ud s,\ud y).
\end{align}
The difficulty is to transform this expression into a form in which the two
$\cA$ integrands have the {\em same} time variable $t_1$.

\medskip
We first show that the moments of $J_N(t_1,t_2)$ are uniformly
bounded, which yields trivially the desired bound \eqref{main-est}, {\bf when
$t_1, t_2$ are far apart}, i.e. $t_2-t_1 > \frac{\varepsilon}{2}$.

\begin{proposition}\label{JN-far}
  For any $p \ge 2$, $N \ge 2$, and $0 < \varepsilon \le t_1 < t_2
  \le T$,
  \begin{equation}\label{JN-bound1}
    \bE|J_N(t_1,t_2)|^p \le C_{p,\varepsilon,T}.
  \end{equation}
  Consequently, if $t_2-t_1 > \frac{\varepsilon}{2}$, for any $\alpha > 0$,
  \begin{equation}\label{JN-bound2}
    \bE|J_N(t_1,t_2)|^p \le C_{p,\alpha,\varepsilon,T}(t_2-t_1)^{p\alpha/2}.
  \end{equation}
\end{proposition}

\begin{proof}
  We use decomposition \eqref{dec-JN}, written as $J_N =
  :J_N^{(2)}-J_N^{(1)}$. By Minkowski's inequality,
  \[
    \|J_N\|_p^2 \le 2 \big(\|J_N^{(1)}\|_p^2+\|J_N^{(2)}\|_p^2\big).
  \]
  We treat the two terms by the same method: for $i = 1, 2$, as in the proof of
  Proposition \ref{IN}, using \eqref{BDG} and Minkowski's inequality,
  \begin{align*}
    \|J_N^{(i)}\|_p^2 & \le \frac{z_p^2 t_i^2}{N \log N} \int_0^{t_1} \int_{\bR}
      \left(\int_0^{N/t_i} \|\cA(t_i,t_i u;s,y)\|_p q_{s}^{t_i,t_i u}(y) \ud u \right)^2
      \ud y\ud s \\
    & \le \frac{C_{p,T}}{N \log N} \int_0^{t_1} \int_{\bR}  \left(\int_0^{N/t_i} q_{s}^{t_i,t_i
      u}(y) \ud u \right)^2 \ud y\ud s.
  \end{align*}
  Recalling definition \eqref{def-q} of $q$, we see that $q_{s}^{t_i,t_i u}(y) =
  p_{\frac{s(t_i-s)}{t_i}}(y-su)$. By Lemma \ref{pcs-lem},
  \[
    \|J_N^{(i)}\|_p^2 \le \frac{C_{p,\varepsilon,T}}{N \log N} \int_0^{t_1} \left(\frac{N}{s}
      \wedge \frac{N^2}{\sqrt{s}}\right) \ud s.
  \]

  We analyze the integral above. If $t_1 > N^{-2}$, we write it as the sum of
  integrals on $[0, N^{-2}]$ and $[N^{-2}, t_1]$ and we bound the integrand
  differently on each region:
  \begin{align*}
    \int_0^{t_1} \left(\frac{N}{s} \wedge \frac{N^2}{\sqrt{s}}\right) \ud s \le
      \int_0^{N^{-2}} \frac{N^2}{\sqrt{s}} \ud s +\int_{N^{-2}}^{t_1} \frac{N}{s}
      \ud s=2N+N \log(t_1 N^2) \le C_T N \log N,
  \end{align*}
  whereas if $t_1 \le N^{-2}$,
  \begin{align*}
    \int_0^{t_1} \left(\frac{N}{s} \wedge \frac{N^2}{\sqrt{s}}\right) \ud s \le
      \int_0^{N^{-2}} \frac{N^2}{\sqrt{s}} \ud s =2N \le \frac{2}{\log 2} N \log N.
  \end{align*}
  This proves \eqref{JN-bound1}. Relation \eqref{JN-bound2} follows immediately
  from \eqref{JN-bound1}.
\end{proof}

It remains to estimate $\bE|J_N(t_1,t_2)|^p$ in the case when {\bf $t_1, t_2$
are close to each other}, i.e. $t_1 < t_2 \le t_1+\frac{\varepsilon}{2}$. This
will require significant effort.

\medskip

We return to decomposition \eqref{dec-JN}. We write the integral on $[0, N/t_1]$
as the sum of integrals on $[0, N/t_2]$ and $[N/t_2, N/t_1]$. Thus, we arrive at
the decomposition:
\begin{equation}\label{dec-JN2}
  J_N(t_1,t_2)=J_N'(t_1,t_2)-J_N''(t_1,t_2),
\end{equation}
where
\begin{align*}
  J_{N}'(t_1,t_2)&=\frac{1}{\sqrt{N \log N }} \int_0^{t_1}\int_{\bR} \left(\int_0^{N/t_2}\Big( t_2
    \cA(t_2,t_2u;s,y) q_s^{t_2,t_2 u}(y)-\right.\\
  & \qquad \qquad \qquad \qquad \qquad \qquad \quad \left.t_1 \cA(t_1,t_1u;s,y) q_s^{t_1,t_1 u}(y)
    \Big) \ud u\right) \xi(\ud s,\ud y) \\
  J_N''(t_1,t_2) &=\frac{1}{\sqrt{N \log N }} \int_0^{t_1}\int_{\bR} \left(\int_{N/t_2}^{N/t_1}t_1
    \cA(t_1,t_1u;s,y) q_s^{t_1,t_1 u}(y) \ud u\right)
    \xi(\ud s,\ud y).
\end{align*}

We study separately $J_N'$ and $J_N''$. The next result examines $J_N''$, which
is simpler, since it does not contain an $\cA$-increment. (The restriction
$t_2-t_1 \le \frac{\varepsilon}{2}$ is not needed for $J_N''$.)

\begin{proposition}\label{prop-JN''}
  For any $p \ge 2$, $N \ge 2$ and $0 < \varepsilon \le t_1 < t_2 \le T$,
  \[
    \bE|J_N''(t_1,t_2)|^p \le C_{p,\varepsilon,T}(t_2-t_1)^{p/2}.
  \]
\end{proposition}

\begin{proof}
  As in the proof of Proposition \ref{IN}, by \eqref{BDG} and Minkowski's
  inequality,
  \begin{align*}
    \|J_N''(t_1,t_2)\|_p^2 & \le \frac{C_{p,T}}{N \log N} \int_0^{t_1} \int_{\bR}
      \left(\int_{N/t_2}^{N/t_1} q_s^{t_1,t_1 u}(y)\ud u \right)^2
      \ud s\ud y.
  \end{align*}
  We denote $\ell = \frac{N}{t_1}-\frac{N}{t_2} > 0$. We claim that:
  \begin{equation}\label{ell-int}
    \int_{\bR}\left(\int_{N/t_2}^{N/t_1} q_{s}^{t_1,t_1u}(y) \ud u \right)^2
      \ud y \le C_T \left\{\frac{\ell}{s} \wedge \frac{\ell^2}{\sqrt{s(t_1-s)}}\right\}.
  \end{equation}

  Indeed, as for \eqref{q-int}, by the semigroup property,
  \begin{align*}
    \int_{\bR}\left(\int_{N/t_2}^{N/t_1} q_{s}^{t_1,t_1u}(y) \ud u \right)^2
      \ud y & = \int_{N/t_2}^{N/t_1} \int_{N/t_2}^{N/t_1} p_{\frac{2s(t_1-s)}{t_1}}
      \big(s(u-v)\big) \ud u \ud v \\
    & \le  \int_{N/t_2}^{N/t_1}  \int_{\bR} p_{\frac{2s(t_1-s)}{t_1}} \big(s(u-v)\big)
      \ud u \ud v =\frac{\ell}{s}.
  \end{align*}
  The second bound follows since $p_{\frac{2s(t_1-s)}{t_1}} \big(s(u-v)\big)
  \le \big(4\pi \frac{s(t_1-s)}{t_1}\big)^{-1/2} \le
  \frac{C_T}{\sqrt{s(t_1-s)}}$.

  It follows that
  \[
    \|J_{N}''(t_1,t_2)\|_p^2 \le \frac{C_{p,T}}{N\log N}
      \int_0^{t_1}\left\{\frac{\ell}{s} \wedge
      \frac{\ell^2}{\sqrt{s(t_1-s)}}\right\}\ud s.
  \]

  We write the integral above as the sum of integrals on $[0, \frac{t_1}{2}]$
  and $[\frac{t_1}{2}, t_1]$, which we call $\cI_1$ and $\cI_2$, respectively.

  We examine $\cI_1$ first. If $\frac{t_1}{2} > \ell^{-2}$, we split the
  integral on $[0, \frac{t_1}{2}]$ as the sum of integrals on $[0, \ell^{-2}]$
  and $[\ell^{-2}, \frac{t_1}{2}]$ and we bound differently the integrand on the
  two regions:
  \begin{align*}
    \cI_1 & \le \int_{0}^{\ell^{-2}} \frac{\ell^2}{\sqrt{s(t_1-s)}} \ud s
      + \int_{\ell^{-2}}^{t_1/2} \frac{\ell}{s}\ud s  \le \sqrt{\frac{2}{\varepsilon}}
      \ell^2 \int_{0}^{\ell^{-2}}\frac{1}{\sqrt{s}}\ud s +\ell \left[\log\left(
      \frac{t_1}{2}\right)-\log \ell^{-2}\right] \\
    &=   2\sqrt{\frac{2}{\varepsilon}} \ell+ \ell \log\left(\frac{t_1 \ell^2}{2}\right) \le
      C_{\varepsilon,T} \ell \big[1+\log (1+\ell)\big] \le C_{\varepsilon,T} \ell \log N,
  \end{align*}
  where for the last inequality, we used the fact that $\ell = N
  \frac{t_2-t_1}{t_1 t_2} \le \frac{N T}{\varepsilon^2}$, and hence
  \[
    1+\log(1+\ell) \le 1+\log \left( 1+ \frac{NT}{\varepsilon^2} \right) \le C_{\varepsilon,T} \log N.
  \]
  If $\frac{t_1}{2} \le \ell^{-2}$, then
  \[
    \cI_1 \le \ell^2\int_0^{t_1/2}\frac{\ud s}{\sqrt{s(t_1-s)}}
    \le\sqrt{\frac2\varepsilon}\,\ell^2\int_0^{\ell^{-2}}s^{-1/2}\,\ud s
    =2\sqrt{\frac2\varepsilon}\,\ell.
  \]
  For $\cI_2$, we have:
  \[
    \cI_2 \le \int_{t_1/2}^{t_1}\frac{\ell}{s} \ud s =\ell \Big(\log t_1-\log \frac{t_1}{2}\Big)=\ell \log 2.
  \]
  Combining these estimates, we obtain that $\cI_1+\cI_2 \le
  C_{p,\varepsilon,T}\ell \log N$. Hence,
  \[
    \|J_{N}''(t_1,t_2)\|_p^2 \le C_{p,\varepsilon,T} \frac{\ell}{N}=C_{p,\varepsilon,T}
      \frac{t_2-t_1}{t_1t_2} \le C_{p,\varepsilon,T} \frac{t_2-t_1}{\varepsilon^2}.
  \]
\end{proof}

The study of $J_N'(t_1,t_2)$ is more involved. Recalling \eqref{Dh}
and \eqref{def-q}, we define
\begin{align}
  \nonumber
  \cR(t,x;s,y)&:=D_{s,y}h(t,x)= \frac{\cG(t,x;s,y)\cG(s,y;0,0)}{\cG(t,x;0,0)}\\
  \label{def-R}
  & =
  q_{s}^{t,x} (y)\frac{\overline{\cG}(t,x;s,y)\overline{\cG}(s,y;0,0)}{\overline{\cG}(t,x;0,0)},
\end{align}
for any $0 < s < t$ and $x, y \in \bR$, so that
\begin{equation}\label{def-R1}
  \bE[\cR(t,x;s,y)|\cF_s]=\cA(t,x;s,y)q_{s}^{t,x}(y).
\end{equation}
It follows that the integrand which appears in $J_N'(t_1,t_2)$ can be written as
\begin{align*}
  & t_2\cA(t_2,t_2u;s,y) q_s^{t_2,t_2 u}(y)-t_1\cA(t_1,t_1u;s,y) q_s^{t_1,t_1 u}(y)=\\
  & \quad \quad \quad \bE[t_2\cR(t_2,t_2u;s,y)-t_1\cR(t_1,t_1 u;s,y)|\cF_s].
\end{align*}

By the generalized H\"older's inequality, an immediate consequence
of \eqref{def-R} and \eqref{barG} is:
\begin{equation}\label{R-bound}
  \big\|\cR(t,x;s,y) \big\|_p \le C_{p,T} \, q_{s}^{t,x} (y) \quad \mbox{for any $p\ge 1$}.
\end{equation}

\underline{We analyze $\cR$}. By \eqref{CK-G}, for any $0 < s < t$ and $x \in
\bR$,
\[
  \int_{\bR}\cR(t,x;s,y)\ud y=\frac{1}{\cG(t,x;0,0)}\int_{\bR}\cG(t,x;s,y)\cG(s,y;0,0)\ud y=1.
\]
By \eqref{CK-G}, $\cR$ also satisfies a Chapman-Kolmogorov relation: for any
$0 < s < r < t$ and $x, y \in \bR$,
\begin{equation}\label{CK-R}
  \cR(t,x;s,y)=\int_{\bR}\cR(t,x;r,z)\cR(r,z;s,y)\ud z.
\end{equation}
Consequently, for any $0 < s < t_1 < t_2$ and $x_1, x_2, y \in \bR$,
\begin{equation}\label{diff-R}
  \cR(t_2,x_2;s,y)-\cR(t_1,x_1;s,y)=\int_{\bR}\big( \cR(t_1,w;s,y)-\cR(t_1,x_1;s,y) \big)
    \cR(t_2,x_2;t_1,w) \ud w.
\end{equation}

Note that in the integral above, the two $\cR$'s have the {\em same} first
variable $t_1$. An increment of this form is examined in the following lemma.
For this lemma, we will use the following ``variance comparison'' property of
the heat kernel:
\begin{equation}\label{heat-prop}
  p_{t}(x) \le \sqrt{c}p_{ct}(x) \quad \mbox{for any $c\ge 1$}.
\end{equation}
Note that \eqref{heat-prop} is a consequence of the following identity:
\begin{equation}\label{heat-id}
  p_{t}(x) = \sqrt{c} e^{-\frac{c-1}{2ct}x^2}p_{ct}(x) \quad \mbox{for any $c>0$}.
\end{equation}

\begin{lemma}\label{lem-R}
  Let $p \ge 1$ and $\alpha \in (0,\frac{1}{2})$. For any $0 < s < t \le T$
  and $x_1, x_2, y \in \bR$,
  \begin{align*}
    & \|\cR(t,x_2;s,y)-\cR(t,x_1;s,y)\|_p \le \\
    & \quad \quad \quad  C_{p,\alpha,T} (t-s)^{-\frac{\alpha}{2}} |x_1-x_2|^{\alpha}
    \Big( p_{\frac{2s(t-s)}{t}} \big(y-\frac{s}{t}x_1\big)+p_{\frac{2s(t-s)}{t}}\big(y-\frac{s}{t}x_2\big) \Big).
  \end{align*}
\end{lemma}

\begin{proof}
  By \eqref{def-R}, $\cR(t,x_2;s,y)-\cR(t,x_1;s,y) = T_1+T_2$, where
  \begin{align*}
    T_1 & =\overline{\cG} (s,y;0,0)  \left(\frac{\overline{\cG}(t,x_2;s,y)}{\overline{\cG}(t,x_2;0,0)}-
    \frac{\overline{\cG}(t,x_1;s,y)}{\overline{\cG}(t,x_1;0,0)}\right)q_{s}^{t,x_2}(y)\\
    T_2 & = \overline{\cG} (s,y;0,0)  \frac{\overline{\cG}(t,x_1;s,y)}{\overline{\cG}(t,x_1;0,0)}
      \big(q_{s}^{t,x_2}(y)-q_{s}^{t,x_1}(y) \big).
  \end{align*}
 By Cauchy--Schwarz, \eqref{barG} and \eqref{ineq-G2},
  \[
    \|T_1\|_p\le C_{p,\alpha,T}|x_1-x_2|^\alpha q_s^{t,x_2}(y).
  \]
  Using $1 \le T^{\alpha/2}(t-s)^{-\alpha/2}$ and \eqref{heat-prop} with $c =
  2$ yields
  \[
    \|T_1\|_p\le C_{p,\alpha,T}(t-s)^{-\alpha/2}|x_1-x_2|^\alpha
    p_{2s(t-s)/t}(y-sx_2/t).
  \]
  Similarly, generalized H\"older's inequality, \eqref{barG} and
  Lemma \ref{increm-q} imply
  \[
    \|T_2\|_p\le C_{p,\alpha,T}(t-s)^{-\alpha/2}|x_1-x_2|^\alpha
    \big(p_{2s(t-s)/t}(y-sx_1/t)+p_{2s(t-s)/t}(y-sx_2/t)\big).
  \]
  For $1 \le p < 2$, the same estimates follow by monotonicity of the moment
  norms from the corresponding estimates at larger moment orders. The triangle
  inequality proves the lemma. 
\end{proof}

The next lemma gives an estimate for the $p$-moment of the integrand in
$J_N'(t_1,t_2)$.

\begin{lemma}\label{tAq-increm}
  Let $p \ge 1$ and $\alpha \in (0,\frac{1}{2})$. Let $0 < t_1 < t_2 \le T$ be
  such that $t_2 < 2t_1$. For any $0 < s \le t_1-(t_2-t_1)$ and $u, y \in \bR$,
  \begin{align*}
    & \big\| t_2 \cA(t_2,t_2 u;s,y) \, q_s^{t_2,t_2 u}(y)- t_1 \cA(t_1,t_1 u;s,y) \, q_s^{t_1,t_1 u}(y)\big\|_p \\
    & \quad \le C_{p,\alpha,T}(t_2-t_1)^{\frac{\alpha}{2}} (t_1-s)^{-\frac{\alpha}{2}} \Big(
      p_{\frac{2s(t_1-s)}{t_1}}(y-su)+{p_{\frac{2s(t_2-s)}{t_2}}(y-su)} \Big),
  \end{align*}
\end{lemma}

\begin{proof}
  Let $0 < s \le t_1-(t_2-t_1)$ and $u, y \in \bR$. Using \eqref{def-R1} and
  Jensen's inequality for conditional expectation,
  \begin{align*}
    & \big\|t_2 \cA(t_2,t_2 u;s,y) \, q_s^{t_2,t_2 u}(y)- t_1 \cA(t_1,t_1 u;s,y) \, q_s^{t_1,t_1 u}(y)\big\|_p \\
    & \quad = \big\| \bE\big[t_2 \cR(t_2,t_2u;s,y)-t_1 \cR(t_1,t_1u;s,y)|\cF_s\big]\big\|_p \\
    & \quad \le \big\|t_2 \cR(t_2,t_2u;s,y)-t_1 \cR(t_1,t_1u;s,y)\big\|_p \\
    & \quad \le t_2 \big\|\cR(t_2,t_2u;s,y)-
      \cR(t_1,t_1u;s,y)\big\|_p+(t_2-t_1)\big\|\cR(t_1,t_1u;s,y)\big\|_p=:A_1+A_2.
  \end{align*}

  We examine $A_2$ first. By the choice of $s$,
  $(t_2-t_1)^{1-\frac{\alpha}{2}}(t_1-s)^{\frac{\alpha}{2}} \le t_1-s \le T$,
  and hence
  \[
    t_2-t_1=(t_2-t_1)^{\frac{\alpha}{2}}(t_1-s)^{-\frac{\alpha}{2}}(t_2-t_1)^{1-\frac{\alpha}{2}}
      (t_1-s)^{\frac{\alpha}{2}} \le T (t_2-t_1)^{\frac{\alpha}{2}}(t_1-s)^{-\frac{\alpha}{2}}.
  \]
  Using \eqref{R-bound} and \eqref{heat-prop} with $c = 2$, we see that
  \[
    \big\|\cR(t_1,t_1u;s,y)\big\|_p \le C_{p,T} \, q_{s}^{t_1,t_1 u}(y)=C_{p,T} \,
      p_{\frac{s(t_1-s)}{t_1}}(y-su) \le C_{p,T} \, p_{\frac{2s(t_1-s)}{t_1}}(y-su).
  \]
  Therefore,
  \begin{equation}\label{A2}
    A_2 \le C_{p,T} \, (t_2-t_1)^{\frac{\alpha}{2}}(t_1-s)^{-\frac{\alpha}{2}}p_{\frac{2s(t_1-s)}{t_1}}(y-su).
  \end{equation}

  To study $A_1$, we use \eqref{diff-R}. By Minkowski's and Cauchy-Schwarz
  inequality,
  \begin{align*}
    A_1 & \le T \int_{\bR}\big\| \cR(t_1,w;s,y)-\cR(t_1,t_1 u;s,y) \big\|_{2p} \big\|\cR(t_2,t_2
      u;t_1,w)\big\|_{2p} \ud w.
  \end{align*}
  We bound these factors using Lemma \ref{lem-R}, respectively \eqref{R-bound}.
  It follows that
  \begin{align}
    \nonumber
    A_1 & \le C_{p,\alpha,T} (t_1-s)^{-\frac{\alpha}{2}}  \int_{\bR}  |w-t_1 u|^{\alpha}
    \Big( p_{\frac{2s(t_1-s)}{t_1}}\big(y-\frac{s}{t_1}w\big)+p_{\frac{2s(t_1-s)}{t_1}} (y-su) \Big)\nonumber\\
    &\qquad\qquad\times p_{\frac{t_1(t_2-t_1)}{t_2}}(w-t_1u)\ud w\nonumber \\
    \label{A1}
    & =:C_{p,\alpha,T} (t_1-s)^{-\frac{\alpha}{2}}  (A_1'+A_1'').
  \end{align}

  We study separately $A_1'$ and $A_1''$. We start with $A_1''$. We have:
  \begin{align*}
    A_1''&=  p_{\frac{2s(t_1-s)}{t_1}} (y-su)  \int_{\bR}  |w-t_1 u|^{\alpha}
    p_{\frac{t_1(t_2-t_1)}{t_2}}(w-t_1u)\ud w \\&=C_{\alpha}p_{\frac{2s(t_1-s)}{t_1}}
      (y-su) \left( \frac{t_1(t_2-t_1)}{t_2}\right)^{\frac{\alpha}{2}},
  \end{align*}
  where $C_{\alpha} = \bE|Z|^{\alpha}$ and $Z \sim N(0,1)$. Here we used the
  fact that
  \[
    \int_{\bR}|x|^{\alpha}p_t(x)\ud x=\bE|X|^{\alpha}=t^{\frac{\alpha}{2}}\bE|Z|^{\alpha},
      \quad \mbox{where $X \sim N(0,t)$}.
  \]
  Using the fact that $\frac{t_1}{t_2} \le 1$, it follows that
  \begin{equation}\label{A1-second}
    A_1'' \le C_{\alpha}(t_2-t_1)^{\frac{\alpha}{2}}\, p_{\frac{2s(t_1-s)}{t_1}} (y-su).
  \end{equation}

  To evaluate $A_1'$, we use the fact that for any $\alpha > 0$, $t > 0$ and $x
  \in \bR$,
  \[
    |x|^{\alpha}p_t(x) \le C_{\alpha} t^{\alpha/2} p_{2t}(x),
  \]
  where $C_{\alpha} = \sqrt{2} \left(\frac{2\alpha}{e} \right)^{\alpha/2}$.
  Indeed, we can write
  \[
    |x|^{\alpha}p_t(x) = \sqrt{2}\, t^{\alpha/2}p_{2t}(x) \left(\frac{|x|}{\sqrt{t}}\right)^{\alpha}
      e^{-\frac{x^2}{4t}} \le \sqrt{2} t^{\alpha/2}p_{2t}(x) \cdot K_{\alpha},
  \] where
  $K_{\alpha}:=\sup_{y>0} y^{\alpha} e^{-y^2/4} =
  (\frac{2\alpha}{e})^{\alpha/2}$. It follows that
  \begin{align*}
    A_1' & \le C_{\alpha} \left( \frac{t_1(t_2-t_1)}{t_2} \right)^{\frac{\alpha}{2}} \int_{\bR}
      p_{\frac{2t_1(t_2-t_1)}{t_2}}(w-t_1 u)
      p_{\frac{2s(t_1-s)}{t_1}}\big(y-\frac{s}{t_1}w\big)\ud w.
  \end{align*}

  To evaluate the last integral, we denote
  \[
    A=\frac{2t_1(t_2-t_1)}{t_2}, \quad \quad B=\frac{2s(t_1-s)}{t_1}, \quad \quad a=\frac{s}{t_1}.
  \]
  Note also that $p_t(x-ay) = \frac{1}{a}p_{t/a^2} \big(\frac{x}{a}-y\big)$ for
  any $t, a > 0$ and $x, y \in \bR$. Hence,
  \begin{align*}
    & \int_{\bR} p_A(w-t_1 u) p_B(y-aw) \ud w=\frac{1}{a}  \int_{\bR}
      p_A\Big(\frac{x}{a}-t_1 u\Big) p_B(y-x)\ud x\\
    & \quad =\int_{\bR} p_{a^2 A}(x-at_1 u) p_B(y-x)\ud x=p_{a^2A+B}(y-at_1 u).
  \end{align*}
  It follows that
  \begin{equation}\label{A1-prime}
    A_1' \le C_{\alpha} (t_2-t_1)^{\frac{\alpha}{2}}p_{a^2A+B}(y-su)=C_{\alpha}
      (t_2-t_1)^{\frac{\alpha}{2}}\, p_{\frac{2s(t_2-s)}{t_2}}(y-su),
  \end{equation}
  since
  \[
    a^2A+B=\frac{2s^2 (t_2-t_1)}{t_1t_2}+\frac{2s(t_1-s)}{t_1}=\frac{2s(t_2-s)}{t_2}.
  \]

  The conclusion follows from \eqref{A2}, \eqref{A1}, \eqref{A1-second}
  and \eqref{A1-prime}.
\end{proof}

\begin{proposition}\label{prop-JN'}
  Let $p \ge 2$, $\alpha \in (0,\frac{1}{2})$ and $N \ge 2$. For any $0 <
  \varepsilon \le t_1 < t_2 \le T$ such that $t_2 -t_1 \le
  \frac{\varepsilon}{2}$,
  \[
    \bE|J_N'(t_1,t_2)|^p \le C_{p,\alpha,\varepsilon,T}(t_2-t_1)^{\alpha p/2}.
  \]
\end{proposition}

\begin{proof}
  As in the proof of Proposition \ref{IN}, by \eqref{BDG} and Minkowski's
  inequality,
  \begin{align*}
    & \|J_N'(t_1,t_2)\|_p^2\\
    &\quad\le\frac{z_p^2}{N\log N}\int_0^{t_1}\int_{\bR}
    \left(\int_0^{N/t_2}
    \left\|\begin{aligned}
      &t_2\cA(t_2,t_2u;s,y)q_s^{t_2,t_2u}(y)\\
      &{}-t_1\cA(t_1,t_1u;s,y)q_s^{t_1,t_1u}(y)
    \end{aligned}\right\|_p \ud u\right)^2\ud y\,\ud s.
  \end{align*}
  We write the integral on $[0, t_1]$ as the sum of integrals on $[0,
  t_1-(t_2-t_1)]$ and $[t_1-(t_2-t_1), t_1]$, and we denote the two terms by
  $T_1$ and $T_2$, so that
  $$
    \|J_N'(t_1,t_2)\|_p^2  \le z_p^2 (T_1+T_2).
  $$

  We examine $T_1$. Note that $t_2 \le t_1 + \frac{\varepsilon}{2} < 2t_1$.
  Hence, we can apply Lemma \ref{tAq-increm} choosing $\alpha \in
  (0,\frac{1}{2})$. Using the inequality $(a+b)^2 \le 2(a^2+b^2)$, it follows
  that
  \begin{align*}
    T_1 & \le C_{p,\alpha,\varepsilon,T} \frac{(t_2-t_1)^{\alpha}}{N \log N}
      \left\{\int_{0}^{t_1-(t_2-t_1)}(t_1-s)^{-\alpha} \int_{\bR}\left[ \int_0^{N/t_2}
      p_{\frac{2s(t_1-s)}{t_1}}(y-su)\ud u \right ]^2 \ud y\ud s
      \right.\\
    & \qquad \qquad \quad \quad \quad \quad \ \ \ +  \left.
      \int_{0}^{t_1-(t_2-t_1)}(t_1-s)^{-\alpha} \int_{\bR}\left[ \int_0^{N/t_2}  {
      p_{\frac{2s(t_2-s)}{t_2}}(y-su)} \ud u \right ]^2 \ud y\ud s
      \right\}\\
    & \le C_{p,\alpha,\varepsilon,T} \frac{(t_2-t_1)^{\alpha}}{N \log N} \int_{0}^{t_1-(t_2-t_1)}(t_1-s)^{-\alpha}
    \left( \frac{N}{s} \wedge \frac{N^2}{\sqrt{s}}\right)\ud s.
  \end{align*}
  For the last line, positivity permits enlargement of the first
  spatial integration interval from $[0, N/t_2]$ to $[0, N/t_1]$; then
  Lemma \ref{pcs-lem} applies to both kernels.

  Note that $t_1-(t_2-t_1) \ge \frac{\varepsilon}{2}$. We split the integral above as the sum of two
  integrals on $[0, \frac{\varepsilon}{2}]$ and $[\frac{\varepsilon}{2},
  t_1-(t_2-t_1)]$. (The second
  interval may be empty.) We denote these two integrals by $A$ and $B$,
  respectively. On the first interval, $t_1-s \ge \varepsilon/2$.
  Hence,
  \begin{align*}
    A & \le \left(\frac{\varepsilon}{2}\right)^{-\alpha}
    \left(\int_0^{\min\{\varepsilon/2,N^{-2}\}}\frac{N^2}{\sqrt{s}}\,\ud s
    +1_{\{\varepsilon/2>N^{-2}\}}
    \int_{N^{-2}}^{\varepsilon/2}\frac{N}{s}\,\ud s\right) \\
    & \le \left(\frac{\varepsilon}{2} \right)^{-\alpha}\left(2N
      + 1_{\{\frac{\varepsilon}{2}>N^{-2}\}}N  \log \frac{\varepsilon N^2}{2}\right) \le
      C_{\alpha,\varepsilon}N \log N \\
    B & \le \int_{\varepsilon/2}^{t_1-(t_2-t_1)}(t_1-s)^{-\alpha}
    \frac{N}{s} \ud s \le
      \frac{2N}{\varepsilon}\int_{\varepsilon/2}^{t_1-(t_2-t_1)}(t_1-s)^{-\alpha} \ud s
      \le \frac{2N}{\varepsilon} \cdot \frac{t_1^{1-\alpha}}{1-\alpha} \le
      C_{\alpha,\varepsilon,T}N.
  \end{align*}
  Therefore,
  \[
    T_1 \le C_{p,\alpha,\varepsilon,T}(t_2-t_1)^{\alpha}.
  \]

  For $T_2$, we use the inequality $\|X+Y\|_{p} \le \|X\|_p+\|Y\|_p$ followed
  by Lemma \ref{A-bound}:
  \begin{align*}
    T_2 &\le \frac{C_{p,T}}{N\log N}
    \sum_{i=1}^2\int_{t_1-(t_2-t_1)}^{t_1}\int_{\bR}
    \left(\int_0^{N/t_2}q_s^{t_i,t_i u}(y)\,\ud u\right)^2\ud y\,\ud s.
  \end{align*}
  We now use again Lemma \ref{pcs-lem} with $c = 1$. We obtain:
  \begin{align*}
    T_2 & \le \frac{C_{p,\varepsilon,T}}{N \log N}\int_{t_1-(t_2-t_1)}^{t_1}\left( \frac{N}{s}
      \wedge \frac{N^2}{\sqrt{s}}\right)\ud s \\
    & \le \frac{C_{p,\varepsilon,T}}{\log N}\int_{t_1-(t_2-t_1)}^{t_1} \frac{1}{s}\ud s
      \le \frac{C_{p,\varepsilon,T}}{\log N} \cdot  \frac{2}{\varepsilon}(t_2-t_1) \le
      C_{p,\varepsilon,T} (t_2-t_1).
  \end{align*}
  For the last line, we used the fact that $s \ge t_1-(t_2-t_1) \ge
  \varepsilon-\frac{\varepsilon}{2} = \frac{\varepsilon}{2}$. 
  Combining the estimates for $T_1$ and $T_2$, we obtain
  \[
    \|J_N'(t_1,t_2)\|_p^2
    \le C_{p,\alpha,\varepsilon,T}
    \big((t_2-t_1)^\alpha+(t_2-t_1)\big)
    \le C_{p,\alpha,\varepsilon,T}(t_2-t_1)^\alpha.
  \]
  Indeed, $t_2-t_1 \le T$ and $0 < \alpha < 1$, so $t_2-t_1 \le
  T^{1-\alpha}(t_2-t_1)^\alpha$. Raising this bound to the power $p/2$ proves
  the proposition.
\end{proof}

\medskip

\begin{proof}[Proof of Theorem \ref{T:FCLT}]
  As mentioned at the beginning of the proof of Section \ref{section-FCLT}, it
  is enough to prove that \eqref{main-est} holds. This follows from
  decompositions \eqref{XN-dec} and \eqref{dec-JN2}, using: a)
  Proposition \ref{IN} for the term $I_N(t_1,t_2)$; b) Proposition \ref{JN-far}
  for the term $J_N(t_1,t_2)$ when $t_2-t_1 > \varepsilon/2$; c)
  Proposition \ref{prop-JN''} for the term $J_N''(t_1,t_2)$ (the restriction
  $t_2-t_1 \le \varepsilon/2$ is not needed); d)
  Proposition \ref{prop-JN'} for the term $J_N'(t_1,t_2)$ when $t_2-t_1 \le
  \varepsilon/2$. Most of the effort is dedicated to d), which is based
  on Lemma \ref{tAq-increm}, which in turn uses Lemma \ref{lem-R}.
\end{proof}


\section*{Acknowledgements}

R.~M.~B. was supported by a grant from the Natural Sciences and
Engineering Research Council of Canada. L.~C. was partially supported by
NSF grant DMS-2443823 and a Collaboration Grant for Mathematicians
(\#959981) from the Simons Foundation. J.~W. was partially supported by a Mitacs
Globalink Research Award Abroad (IT52612), which funded his visit to Auburn
University.

\appendix

\section{Auxiliary results}
\label{S:appendix}

\subsection{Properties of \texorpdfstring{$p_t(x)$ and $q_s^{t,x}(y)$}{p and q}}
\label{S:heat-kernel-estimates}

In this section, we include some auxiliary results about the increments of $p$
and $q$.

\begin{lemma}\label{lem-p1}
  Let $\alpha \in (0,1]$ be arbitrary. For any $t > 0$ and $x, y \in \bR$,
  \[
    \int_{\bR}\big|p_t(x-z)-p_t(y-z)\big|^2 \ud z \le \frac{4^{-\alpha}}{\sqrt{\pi}} \,
      t^{-\frac{1}{2}-\alpha}|x-y|^{2\alpha}.
  \]
\end{lemma}

\begin{proof}
  By the semigroup property of the heat kernel, we obtain
  \begin{align*}
    \int_{\bR}\big|p_t(x-z)-p_t(y-z)\big|^2 \ud z
    =2p_{2t}(0)-2p_{2t}(x-y)
    =\frac{1}{\sqrt{\pi t}}
    \left(1-\exp\left(-\frac{|x-y|^2}{4t}\right)\right).
  \end{align*}
  For any $r \ge 0$ and $\alpha \in (0,1]$, we have $1-e^{-r} \le 1\wedge r
  \le r^{\alpha}$. Therefore,
  \[
    \int_{\bR}\big|p_t(x-z)-p_t(y-z)\big|^2 \ud z
    \le \frac{1}{\sqrt{\pi t}}
    \left(\frac{|x-y|^2}{4t}\right)^{\alpha}
    =\frac{4^{-\alpha}}{\sqrt{\pi}}\,
    t^{-\frac12-\alpha}|x-y|^{2\alpha}.
  \]
\end{proof}

\begin{lemma}\label{lem-p2}
  Let $\alpha \in (0,1]$. For any $t > 0$ and $x, y, z \in \bR$,
  \[
    \big|p_{t}(x-z)-p_t(y-z)\big| \le C_{\alpha} \, t^{-\frac{\alpha}{2}}|x-y|^{\alpha}\big(
      p_{2t}(x-z)+p_{2t}(y-z)\big).
  \]
\end{lemma}

\begin{proof} {We refer to the proof of Lemma 3.1 of \cite{chen-huang19}.}
  
\end{proof}

Recall definition \eqref{def-q} of $q$. The next result gives an estimate for
the increments of $q$.

\begin{lemma}\label{increm-q}
  Let $\alpha \in (0,1]$. For any $0 < s < t$ and $x_1, x_2, y \in \bR$,
  \[
    \big|q_{s}^{t,x_2}(y)-q_{s}^{t,x_1}(y)\big| \le C_{\alpha} (t-s)^{-\frac{\alpha}{2}} |x_1-x_2|^{\alpha}
    \Big( p_{\frac{2s(t-s)}{t}}\big(y-\frac{s}{t}x_1\big)+p_{\frac{2s(t-s)}{t}}\big(y-\frac{s}{t}x_2\big) \Big).
  \]
\end{lemma}

\begin{proof}
  We apply Lemma \ref{lem-p2}.
  It follows that
  \begin{align*}
    &\big|q_s^{t,x_2}(y)-q_s^{t,x_1}(y)\big| = \left|p_{\frac{s(t-s)}{t}} \Big(y-\frac{s}{t}x_1
      \Big)-p_{\frac{s(t-s)}{t}} \Big(y-\frac{s}{t}x_2 \Big)\right| \\
    &\quad\le C_{\alpha}
    \left(\frac{s(t-s)}{t}\right)^{-\frac{\alpha}{2}}
    \left(\frac{s}{t}|x_1-x_2|\right)^{\alpha}
    \left(
    p_{\frac{2s(t-s)}{t}}\left(y-\frac{s}{t}x_1\right)
    +p_{\frac{2s(t-s)}{t}}\left(y-\frac{s}{t}x_2\right)
    \right).
  \end{align*}
  The conclusion follows since
  \[
    \left(\frac{s(t-s)}{t}\right)^{-\frac{\alpha}{2}}
    \left(\frac{s}{t}\right)^{\alpha}
    =\left(\frac{s}{t}\right)^{\frac{\alpha}{2}}
    (t-s)^{-\frac{\alpha}{2}}
    \le(t-s)^{-\frac{\alpha}{2}}.
  \]
\end{proof}

\begin{lemma}\label{pcs-lem}
  Let $c > 0$ and $N \ge 1$. For any $0 < \varepsilon \le t_1 < t_2 \le T$
  and $s \in (0,t_1)$,
  \[
    \int_{\bR} \left(\int_{0}^{N/t_i} p_{\frac{cs (t_i-s)}{t_i}}(y-su)\ud u \right)^2
      \ud y \le C_{c,\varepsilon, T} \left( \frac{N}{s} \wedge
      \frac{N^2}{\sqrt{s}}\right), \quad \mbox{for} \ i=1,2.
  \]
\end{lemma}

\begin{proof}
  We denote the integral on the left side by $I$. Using the semigroup property,
  for $i = 1, 2$,
  \begin{align}
    \nonumber
    I=\int_0^{N/t_i}\int_0^{N/t_i}
    p_{\frac{2cs(t_i-s)}{t_i}}\big(s(u-v)\big)\ud u\,\ud v
    &\le\int_0^{N/t_i}\int_{\bR}
    p_{\frac{2cs(t_i-s)}{t_i}}\big(s(u-v)\big)\ud v\,\ud u\\
    \label{bound1}
    &=\frac{N}{t_i s} \le\frac{N}{\varepsilon s}.
  \end{align}

  This gives the first bound. To obtain the second bound, suppose first that $s
  \in (0,\frac{\varepsilon}{2})$. Then $t_2-s > t_1-s >
  t_1-\frac{\varepsilon}{2} \ge \frac{\varepsilon}{2}$. Bounding the integrand
  by its maximum, we have:
  \[
    p_{\frac{2cs(t_i-s)}{t_i}}\big(s(u-v)\big) \le
      p_{\frac{2cs(t_i-s)}{t_i}}(0)=\frac{1}{\sqrt{4\pi c}} \sqrt{\frac{t_i}{s(t_i-s)}} \le
      \frac{1}{\sqrt{4\pi c}}  \sqrt{\frac{T}{\varepsilon/2}} \cdot \frac{1}{\sqrt{s}},
  \]
  and
  \[
    I \le \left( \frac{N}{t_i}\right)^2 \frac{1}{\sqrt{4\pi c}}  \sqrt{\frac{T}{\varepsilon/2}}
      \cdot \frac{1}{\sqrt{s}} \le  \left( \frac{N}{\varepsilon}\right)^2 \frac{1}{\sqrt{4\pi c}}
      \sqrt{\frac{T}{\varepsilon/2}} \cdot \frac{1}{\sqrt{s}}.
  \]
  If $s \in [\frac{\varepsilon}{2},t_1)$, then $\frac{1}{s} \le
  \sqrt{\frac{2}{\varepsilon}} \frac{1}{\sqrt{s}}$, and using \eqref{bound1}, we
  get:
  \[
    I \le \frac{N}{\varepsilon s} \le \frac{N^2}{\varepsilon s}\le \frac{N^2}{\varepsilon }
      \sqrt{\frac{2}{\varepsilon}} \frac{1}{\sqrt{s}}.
  \]

\end{proof}

\subsection{Increments of \texorpdfstring{$\overline{\cG}$}{normalized G}}
\label{S:green-kernel-increments}

In this section, we present a lemma about the increments of $\overline{\cG}$.

\begin{lemma}\label{increm-G}
  Let $p \ge 2$ and $\alpha \in (0,\frac{1}{2})$. For any $0 \le s < t \le T$
  and $x_1, x_2, y \in \bR$,
  \begin{equation}\label{ineq-G1}
    \|\overline{\cG}(t,x_2;s,y)-\overline{\cG}(t,x_1;s,y)\|_p \le C_{p,T}\,
      \frac{2^{-\alpha}\pi^{1/4}}{\sqrt{\cos(\pi\alpha)}}\,(t-s)^{\frac14-\frac{\alpha}{2}}
      |x_1-x_2|^{\alpha},
  \end{equation}
  and consequently,
  \begin{equation}\label{ineq-G2}
    \left\|\frac{\overline{\cG}(t,x_2;s,y)}{\overline{\cG}(t,x_2;0,0)}-
    \frac{\overline{\cG}(t,x_1;s,y)}{\overline{\cG}(t,x_1;0,0)}\right\|_p \le C_{\alpha,p,T}\, |x_1-x_2|^{\alpha}.
  \end{equation}
\end{lemma}

\begin{proof}
  We first prove \eqref{ineq-G1}. Dividing \eqref{G-eq} by $p_{t-s}(x-y)$, and
  using \eqref{def-q}, we obtain:
  \begin{align*}
    \overline{\cG}(t,x;s,y)
    &=1+\int_s^t\int_{\bR}
    \frac{p_{t-r}(x-z)p_{r-s}(z-y)}{p_{t-s}(x-y)}
    \overline{\cG}(r,z;s,y)\xi(\ud r,\ud z)\\
    &=1+\int_s^t\int_{\bR}
    q_{r-s}^{t-s,x-y}(z-y)\overline{\cG}(r,z;s,y)\xi(\ud r,\ud z).
  \end{align*}
  By inequality \eqref{BDG}, relation \eqref{barG}, and Lemma \ref{lem-p1},
  \begin{align*}
    &\|\overline{\cG}(t,x_2;s,y)-\overline{\cG}(t,x_1;s,y)\|_p^2\\
    &\quad\le C_{p,T}\int_s^t\int_{\bR}
    \left|q_{r-s}^{t-s,x_2-y}(z-y)-q_{r-s}^{t-s,x_1-y}(z-y)\right|^2 \ud z\,\ud r\\
    &\quad\le C_{p,T}\, \frac{4^{-\alpha}}{\sqrt{\pi}} \,|x_1-x_2|^{2\alpha}
    \int_s^t\left(\frac{r-s}{t-s}\right)^{2\alpha}
    \left(\frac{(r-s)(t-r)}{t-s}\right)^{-\frac12-\alpha}\ud r\\
    &\quad=C_{p,T}\,  \frac{4^{-\alpha}}{\sqrt{\pi}} \,|x_1-x_2|^{2\alpha}(t-s)^{\frac12-\alpha}
    \int_s^t(r-s)^{\alpha-\frac12}(t-r)^{-\alpha-\frac12}\ud r\\
    &\quad= C_{p,T}\, \frac{4^{-\alpha}}{\sqrt{\pi}}
      \,\Gamma\left(\alpha+\frac12\right)\Gamma\left(\frac12-\alpha\right)|x_1-x_2|^{2\alpha}(t-s)^{\frac12-\alpha}\\
    &\quad =C_{p,T}\,\frac{4^{-\alpha}\sqrt{\pi}}{\cos(\pi\alpha)}\,|x_1-x_2|^{2\alpha}(t-s)^{\frac12-\alpha}.
  \end{align*}
  The Beta integral is independent of $t-s$, and the reflection formula
  $\Gamma(z)\Gamma(1-z) = \pi/\sin(\pi z)$ with $z = \alpha+\frac12$ evaluates
  it as $\pi/\cos(\pi\alpha)$. Taking square roots, we obtain \eqref{ineq-G1}.

  For \eqref{ineq-G2}, we write
  \begin{align*}
    &\frac{\overline{\cG}(t,x_2;s,y)}{\overline{\cG}(t,x_2;0,0)}
    -\frac{\overline{\cG}(t,x_1;s,y)}{\overline{\cG}(t,x_1;0,0)}\\
    &\quad =\frac{\overline{\cG}(t,x_2;s,y)-\overline{\cG}(t,x_1;s,y)}
    {\overline{\cG}(t,x_2;0,0)}
    +\overline{\cG}(t,x_1;s,y)\,
    \frac{\overline{\cG}(t,x_1;0,0)-\overline{\cG}(t,x_2;0,0)}
    {\overline{\cG}(t,x_1;0,0)\,\overline{\cG}(t,x_2;0,0)},
  \end{align*}
  and apply generalized H\"older's inequality. Since \eqref{barG} holds for
  every $k \in \bR$, the reciprocals $\overline{\cG}(t,x_i;0,0)^{-1}$ have
  finite moments of every order. The first term is estimated by \eqref{ineq-G1},
  and the second by \eqref{ineq-G1} with $s = 0$, which contributes
  $t^{\frac14-\frac{\alpha}{2}}$. Both powers are at most
  $T^{\frac14-\frac{\alpha}{2}}$. This proves Lemma \ref{increm-G}.
\end{proof}

\subsection{Covariance of the normalized stochastic heat equation}
\label{S:normalized-SHE-covariance}
The covariance asymptotic below follows from~\cite[Remark 4.4]{gu-pu}, applied
to the identity function. The variance asymptotic is~\cite[Proposition
4.1]{CKNP22-JFA}. We include a short direct proof for completeness.

\begin{proposition}[Normalized-SHE
  covariance]\label{P:normalized-SHE-covariance}
  Recall $\rho_t^{\overline{Z}}(x)$ is defined in \eqref{E:rho}. For every fixed
  $t > 0$, the normalized solution satisfies
  \begin{align*}
    \rho_t^{\overline{Z}}(x)
    \sim \frac{t}{|x|} \quad \text{as $|x| \to \infty$} \quad \text{and} \quad
    \mathrm{Var}(\overline{S}_N(t))
    \sim 2tN\log N \quad \text{as $N \to \infty$}.
  \end{align*}
\end{proposition}

\begin{proof}
  Put $m(s) = \bE|\overline{Z}(s,0)|^2$. The identity $\overline{Z}(s,x) =
  \overline{\cG}(s,x;0,0)$ and Theorem 1.3 of \cite{chen-jimenez26}, with noise
  strength one, imply $m(s) \to 1$ as $s \downarrow 0$ and $M_t:=\sup_{0 < s
  \le t}m(s) < \infty$. It\^o's isometry and the Gaussian semigroup identity
  give the first equality below. For $x > 0$, Gaussian scaling and the
  substitutions $v = t(t-s)/s$ and $v = x^2w$ give the second:
  \begin{align*}
    \rho_t^{\overline{Z}}(x)
    &=\int_0^t m(s)p_{2s(t-s)/t}(sx/t)\,\ud s,\\
    x\rho_t^{\overline{Z}}(x)
    &=t\int_0^\infty\frac{x^2}{t+x^2w}
      m\left(\frac{t^2}{t+x^2w}\right)p_{2w}(1)\,\ud w.
  \end{align*}
  The second integrand converges to $t p_{2w}(1)/w$ and is bounded by $tM_t
  p_{2w}(1)/w$. This majorant is integrable, with $\int_0^\infty
  p_{2w}(1)w^{-1}\,\ud w = 1$, as the substitution $z = 1/(2\sqrt{w})$ shows.
  Dominated convergence and evenness prove the covariance asymptotic.

  By stationarity, $\mathrm{Var}(\overline{S}_N(t)) =
  2\int_0^N(N-x)\rho_t^{\overline{Z}}(x)\,\ud x$. The covariance is nonnegative
  by its integral formula and bounded by $m(t)-1$ by Cauchy--Schwarz, so a fixed
  interval $[0, R]$ contributes $O(N)$. For $\delta \in (0, t)$, choose $R > 0$
  such that the remaining tail lies between $(t-\delta)/x$ and $(t+\delta)/x$.
  Since $\int_R^N(N-x)x^{-1}\,\ud x = N\log N+O_R(N)$, the variance divided by
  $N\log N$ has liminf at least $2(t-\delta)$ and limsup at most $2(t+\delta)$.
  Letting $\delta \downarrow 0$ proves the second assertion. 
\end{proof}

\end{document}